\documentclass[12pt,reqno]{amsart}
\usepackage{amsmath,amsthm,amsfonts,amssymb,times}
\usepackage{verbatim}
\usepackage{url}
\setbox0=\hbox{$+$}
\newdimen\plusheight
\plusheight=\ht0
\def\+{\;\lower\plusheight\hbox{$+$}\;}

\setbox0=\hbox{$-$}
\newdimen\minusheight
\minusheight=\ht0
\def\-{\;\lower\minusheight\hbox{$-$}\;}

\setbox0=\hbox{$\cdots$}
\newdimen\cdotsheight
\cdotsheight=\plusheight
\def\cds{\lower\cdotsheight\hbox{$\cdots$}}

\makeatletter
\def\leqalignno#1{\displ@y \tabskip\z@ plus\@ne fil
  \halign to\displaywidth{\hfil$\@lign\displaystyle{##}$\tabskip\z@skip
    &$\@lign\displaystyle{{}##}$\hfil\tabskip\z@ plus\@ne fil
    &\kern-\displaywidth\rlap{$\@lign\hbox{\rm##}$}\tabskip\displaywidth\crcr
    #1\crcr}}
\makeatother

\newcommand{\eb}{\begin{equation}}
\newcommand{\ee}{\end{equation}}

\newcommand{\df}{\dfrac}

 \renewcommand{\a}{\alpha}

\renewcommand{\l}{\lambda}

\renewcommand{\(}{\left\(}
\renewcommand{\)}{\right\)}
\renewcommand{\[}{\left\[}
\renewcommand{\]}{\right\]}
\renewcommand{\i}{\infty}

\numberwithin{equation}{section}
 \theoremstyle{plain}
\newtheorem{theorem}{Theorem}[section]
\newtheorem{lemma}[theorem]{Lemma}
\newtheorem{corollary}[theorem]{Corollary}

\newtheorem{example}[]{Example}

   \makeatletter
\def\proof{\@ifnextchar[{\@oproof}{\@nproof}}
\def\@oproof[#1][#2]{\trivlist\item[\hskip\labelsep\textit{#2 Proof of\
#1.}~]\ignorespaces}
\def\@nproof{\trivlist\item[\hskip\labelsep\textit{Proof.}~]\ignorespaces}

\makeatother

\begin{document}

\textbf{\title[Generalizations of the Andrews-Yee identities]{Generalizations of the Andrews-Yee identities associated with the mock theta functions $\omega(q)$ and $\nu(q)$}
\author{Bruce C.~Berndt, Atul Dixit, Rajat Gupta}
\address{Department of Mathematics, University of Illinois, 1409 West Green
Street, Urbana, IL 61801, USA}\email{berndt@illinois.edu}
\address{Department of Mathematics, Indian Institute of Technology, Gandhinagar, Palaj, Gandhinagar 382355, Gujarat, India}\email{adixit@iitgn.ac.in}
\address{Department of Mathematics, Indian Institute of Technology, Gandhinagar, Palaj, Gandhinagar 382355, Gujarat, India}\email{rajat\_gupta@iitgn.ac.in}}

\begin{abstract}
George Andrews and Ae Ja Yee recently established beautiful results involving bivariate generalizations of the third order mock theta functions $\omega(q)$ and $\nu(q)$, thereby extending their earlier results with the second author. Generalizing the Andrews-Yee identities for trivariate generalizations of these mock theta functions remained a mystery, as pointed out by Li and Yang in their recent work. We partially solve this problem and generalize these identities. Several new as well as well-known results are derived. For example, one of our two main theorems gives, as a corollary, a special case of Soon-Yi Kang's three-variable reciprocity theorem. A relation between a new restricted overpartition function $p^{*}(n)$ and a weighted partition function $p_*(n)$ is obtained from one of the special cases of our second theorem.
\end{abstract}
\maketitle
\section{Introduction}
The third order mock theta functions $\omega(q)$ and $\nu(q)$ defined by Watson in his presidential address to the London Mathematical Society \cite[p.~62]{watsonaddress}, and which are actually present in Ramanujan's Lost Notebook \cite[p.~31]{lnb}, \cite[p.~15]{AB5}, are defined by
\begin{align}
\omega(q)&:=\sum_{n=0}^{\infty} \frac{q^{2n^2+2n}}{(q;q^2)^2_{n+1}},\\
\nu(q)&:=\sum_{n=0}^{\infty} \frac{q^{n^2+n}}{(-q;q^2)_{n+1}},
\end{align}
where throughout this article $|q|<1$, and where we use the familiar notation
\begin{align*}
(a)_n &:=(a;q)_{n}:=\prod^{n-1}_{m=0}(1-aq^m),    \qquad  (a;q)_{0}:=1,\\
\intertext{and}
(a)_{\infty}&:=(a;q)_{\infty}:=\lim_{n\to\infty}(a;q)_{n}, \qquad |q|<1.
\end{align*}
Also, define
\begin{equation*}
(a_1,a_2,\dots, a_m;q)_{\infty}
:=(a_1;q)_{\infty}(a_2;q)_{\infty}\cdots(a_m;q)_{\infty}.
\end{equation*}

Recently, G.~E.~Andrews, the second author, and A.~J.~Yee \cite{ady1} introduced new restricted partition functions $p_{\omega}(n)$ and $p_{\nu}(n)$, which are intimately connected, respectively, to $\omega(q)$ and $\nu(q)$. More precisely, if

   \begin{center}
    \begin{align*}
  p_{\omega}(n):=&\text{the number of partitions of $n$ such that all}\\
                 &\text{odd parts are less than twice the smallest part,}\\
   p_{\nu}(n):=& \text{the number of partitions of $n$ into distinct non-negative parts}\\
               & \text{such that all odd parts are less than twice the smallest part},
   \end{align*}
   \end{center}
    they proved that  \cite[Theorems 3.1, 4,1]{ady1}
\begin{align}
\sum_{n=1}^{\infty}p_{\omega}(n)q^n&=\sum_{n=1}^{\infty} \frac{q^n}{(q^{n};q)_{n+1} (q^{2n+2};q^2)_{\infty}}
=q\, \omega(q),\label{pwn}\\
\sum_{n=0}^{\infty}p_{\nu}(n)q^n&=\sum_{n=0}^{\infty}q^n(-q^{n+1};q)_{n} (-q^{2n+2};q^2)_{\infty}
=\nu(-q).\label{pnn}
\end{align}
(The word `non-negative' is missing in the definition in \cite{ady1}, but it ought to be present since the sum on the right-hand side of \eqref{pnn} starts from $n=0$. This is also pointed out in \cite[Equation (5.8)]{andrewsevenbelowodd}.)
	In the same paper, they also obtained new analogues of Euler's pentagonal number theorem by proving that \cite[Theorems 5.1, 5.3]{ady1}
\begin{align}
\sum_{n=1}^{\infty} \frac{q^n}{(-q^n;q)_{n+1}(-q^{2n+2};q^2)_{\infty}}&=\sum_{j=0}^{\infty} (-1)^j q^{6j^2+4j+1}(1+q^{4j+2}),\label{epnt1}\\
\sum_{n=0}^{\infty} q^n (q^{n+1};q)_n (q^{2n+2};q^2)_{\infty} &=\sum_{j=0 }^{\infty} (-1)^j q^{j(3j+2)} (1+q^{2j+1}).\label{epnt2}
\end{align}

Andrews \cite[Section 3]{geaquart} defined two-variable generalizations of several of the third order mock theta functions including those of $\omega(q)$ and $\nu(q)$. In particular, he defined
\begin{align}
\omega(z; q)&:=\sum_{n=0}^{\infty} \frac{z^nq^{2n^2+2n}}{(q;q^2)_{n+1}(zq;q^2)_{n+1}},\label{omegaqz}\\
\nu(z; q)&:=\sum_{n=0}^{\infty} \frac{q^{n^2+n}}{(-zq;q^2)_{n+1}}.
\end{align}
(Note that $\omega(1;q)=\omega(q)$ and $\nu(1;q)=\nu(q)$.)
In a subsequent paper \cite{geaquart2}, Andrews proved that
\begin{align*}
\omega(z;q)&=\sum_{n=1}^{\infty}\frac{z^{n-1}q^{n-1}}{(q;q^2)_{n}},\\
\nu(z;q)&=\sum_{n=0}^{\infty}(q/z;q^2)_n(-zq)^n.
\end{align*}

Andrews and Yee \cite{andrewsyeemock} recently considered another generalization of $\nu(q)$, namely,
\begin{align}\label{nuqz2}
\nu_1(z; q):=\sum_{n=0}^{\infty} \frac{z^nq^{n^2+n}}{(-q;q^2)_{n+1}}.
\end{align}
(Observe that $\nu_1(1;q)=\nu(q)$.)
They then obtained the following beautiful two-variable generalizations \cite[Theorem 1]{andrewsyeemock} of \eqref{pwn}, \eqref{pnn}, \eqref{epnt1} and \eqref{epnt2}:
\begin{align}
\sum_{n=1}^{\infty} \frac{q^n}{(zq^{n};q)_{n+1} (zq^{2n+2};q^2)_{\infty}}
&=q\, \omega(z; q),\label{pwnz}\\
\sum_{n=0}^{\infty}q^n(-zq^{n+1};q)_{n} (-zq^{2n+2};q^2)_{\infty}
&=\nu_1(z; -q).\label{pnnz}
\end{align}
Andrews and Yee  established \eqref{pwnz} and \eqref{pnnz} by proving that in each case the coefficients of $z^{N}$, $N\geq 1,$ in the power series expansions of the left- and right-hand sides are equal. 
Andrews and Yee  also expressed the difficulty in obtaining  bijective proofs of their results.

Recently, F.~Z.~K.~Li and J.~Y.~X.~Yang \cite{liyangmock} overcame this difficulty and provided bijective proofs of \eqref{pwnz} and \eqref{pnnz}. In the same paper, Li and Yang also gave three-variable generalizations of $\omega(z;q)$ and $\nu(z;q)$, namely,
\begin{align}
\omega_0(y,z; q)&:=\sum_{n=0}^{\infty} \frac{y^{n}z^nq^{2n^2+2n}}{(yq;q^2)_{n+1}(zq;q^2)_{n+1}},\label{omegaqzy}\\
\nu_0(y, z; q)&:=\sum_{n=0}^{\infty} \frac{y^{n}z^nq^{n^2+n}}{(yq;q^2)_{n+1}}.\label{nuqzy}
\end{align}
It might be noted that Y.-S.~Choi \cite{choi} had previously defined three-variable generalizations of most of Ramanujan's mock theta functions as well as of some of Andrews' two-variable generalizations. Choi's  three-variable generalizations are at the same level of generality as those of Li and Yang. In fact,
\begin{align}\label{equivalence}
\omega_0\left(\frac{z^2}{q^2},\frac{\a^2z^2}{q^4}; q\right)&=\frac{q^4}{z^4}\omega_{*}(\alpha,z;q), \nonumber\\
	\nu_0\left(-\frac{\a^2z^2}{q^4}, -\frac{q^2}{\a^2}; q\right)&=\nu_{*}(\alpha,z;q),
	\end{align}
	where $\omega_{*}(\alpha,z;q)$ and $\nu_{*}(\alpha,z;q)$ are Choi's trivariate generalizations of the mock theta functions defined by \cite[p.~347]{choi}
	\begin{align}\label{choiomeganu}
\omega_{*}(\alpha,z;q)&:=\sum_{n=0}^{\infty}\frac{q^{2(n-1)^2-6}\a^{2n}z^{4(n+1)}}{(z^2/q;q^2)_{n+1}(\a^2z^2/q^3;q^2)_{n+1}},\nonumber\\
\nu_{*}(\alpha,z;q)&:=\sum_{n=0}^{\infty}\df{q^{n^2-n}z^{2n}}{(-\alpha^2z^2/q^3;q^2)_{n+1}}.
\end{align}
Observe that Li and Yang's definitions \eqref{omegaqzy} and \eqref{nuqzy} are symmetric in $y$ and $z$.  (We have used a slightly different notation for Li and Yang's definitions in order to distinguish them from Choi's definitions in \eqref{choiomeganu}.)

The primary goal of our paper is to generalize the Andrews-Yee identities in \eqref{pwnz} and \eqref{pnnz}. We now explain why it is important to seek these generalizations. Note that $\omega_0(y,z; q)$, defined in \eqref{omegaqzy} can be easily written in terms of the function
\begin{equation*}
G(a, b; q):=\sum_{n=0}^{\infty}\frac{a^nb^nq^{n^2}}{(aq)_n(bq)_n}.
\end{equation*}
More precisely,
\begin{equation}\label{omegaG}
\omega_0(y, z;q)=\frac{1}{(1-yq)(1-zq)}G(yq,zq;q^2).
\end{equation}
In his Lost Notebook \cite{lnb}, Ramanujan obtained beautiful identities for $G(a, b;q)$ \cite[Chapter II]{yesto}. See \cite[p.~202]{lnb}, \cite[p.~89]{AB2} and \cite[p.~13, Equation (12.3)]{fine} for additional representations.

Also, it is known that \cite[p.~351]{choi}
\begin{equation*}
G(x, x^{-1}q; q)=(1-x)(1-x^{-1}q)g_3(x,q),
\end{equation*}
where $g_3(x,q)$ is the universal mock theta function defined by
\begin{equation*}
g_3(x,q):=\sum_{n=0}^{\infty}\dfrac{q^{n(n+1)}}{(x;q)_{n+1}(q/x;q)_{n+1}}.
\end{equation*}
Thus obtaining a generalization of the identity \eqref{pwnz} of Andrews and Yee for $\omega_0(y,z; q)$  might be very crucial in that it may lead to partition-theoretic interpretations of mock theta functions of odd order, which are special cases of $g_3(x,q)$, as explicated in \cite[Section 5]{gordmac}. However, at the end of their paper \cite{liyangmock}, Li and Yang surmise that finding generalizations of \eqref{pwnz} and \eqref{pnnz} for
$\omega_0(y,z; q)$ and $\nu_0(y, z; q)$ remains mysterious. While we have not been able to fully resolve this mystery, we have partially resolved it. Before we give our main results in this direction, it is important to note the possible obstacles.

Observe that the placing of the extra variable $z$ on the left-hand sides of the Andrews-Yee identities \eqref{pwnz} and \eqref{pnnz} was natural in view of the goal of simultaneously generalizing \eqref{pwn}, \eqref{pnn}, \eqref{epnt1} and \eqref{epnt2}. The corresponding right-hand sides too turned out to be natural generalizations of the mock theta functions $\omega(q)$ and $\nu(q)$.

However, the situation is not so obvious while generalizing the Andrews-Yee identities with an extra variable, say, $\a$, because of its possible appearance at unbeknownst places. First of all, there are two ways of generalizing \eqref{pwnz} and \eqref{pnnz} -- one, beginning with the right-hand sides consisting of $\omega(z;q)$ and $\nu(z;q)$, and second, beginning with the left-hand sides which generate the restricted partition functions $p_{\omega}(n)$ and $p_{\nu}(n)$. As far as the latter is concerned, there have already been generalizations of \eqref{pwnz} and \eqref{pnnz}. For example, Wang and Ma \cite[Theorem 1.2]{wangma} have obtained nice identities for
\begin{align*}
&\sum_{n=1}^{\infty} \frac{y^{n-1}}{(zq^{n};q)_{n+1} (zq^{2n+2};q^2)_{\infty}},\\
&\sum_{n=0}^{\infty} y^n(-zq^{n+1};q)_{n+1} (-zq^{2n+2};q^2)_{\infty},
\end{align*}
which, for $y=q$, (essentially) reduce to \eqref{pwnz} and \eqref{pnnz}. However, while these are interesting identities in the theory of basic hypergeometric series, with $y^n$ replacing $q^n$ at each occurrence, the partition-theoretic notion of $n$ representing the smallest part gets compromised, and hence combinatorially it is difficult to interpret them.

Our goal here is to retain the notion of $n$ denoting the smallest part of a partition and then to generalize \eqref{pwnz} and \eqref{pnnz}. In this regard, the three-variable generalizations of the mock theta functions $\omega(q)$ and $\nu(q)$ that we investigate are defined by
	\begin{align}
	\omega_{1}(\a, z; q):&=\sum_{n=0}^{\infty}\frac{q^{2n}}{(-zq;q^2)_{n+1}(zq/\a;q^{2})_{n+1}},\label{2b}\\
	\nu(\a, z; q) :&=
\sum_{n=0}^{\infty}\frac{\a^{n}q^{n^{2}+n}}{(-zq;q^2)_{n+1}}.\label{1}
	\end{align}
Observe that
\begin{align}
\omega_1(1, z; q)&=\omega(z^2;q^2),\label{omega1omega}\\
	\nu(1, z; q) &=\nu(z;q),\nonumber\\
	\nu(z, 1; q)&=\nu_1(z;q),
  	\end{align}
	where $\omega(z;q)$, $\nu(z;q)$, and  $\nu_1(z;q)$ are defined in \eqref{omegaqz}, \eqref{nuqz2}, and \eqref{omegaqzy}, respectively, and where \eqref{omega1omega} is easily derived by using Andrews' identity \cite[p.~24]{yesto}
	\begin{equation*}
	\sum_{n=0}^{\infty}\frac{a^nb^nq^{n^2+2n}}{(aq)_{n+1}(bq)_{n+1}}=\sum_{n=0}^{\infty}\frac{a^nq^n}{(bq)_{n+1}}.
	\end{equation*}
We introduce the two trivariate generalizations \eqref{2b} and \eqref{1} of $\omega(q)$ and $\nu(q)$, respectively, because they are linked by a nice functional relation given below.

\begin{theorem}\label{newrelation}
For $|q|<1$,
\begin{align}\label{newrelationeqn}
\nu{(\a, z; q)} =\df{1}{(-zq;q^2)_{\infty}(zq/\a;q^{2})_{\infty}}\sum_{n=0}^{\infty}\a^nq^{n^2+n} -\frac{zq}{\a}\omega_{1}(\a, z; q).
\end{align}
\end{theorem}
Note that using
\begin{equation*}
\sum_{n=0}^{\infty}q^{n^2+n}=(q^2;q^2)_{\infty}(-q^2;q^2)_{\infty}^{2},
\end{equation*}
which is an easy consequence of the Jacobi triple product identity (see \eqref{jtpi} below), we easily obtain
the special case $\a=1$ of the functional relation obtained by Andrews in \cite[p.~78, Equation (3c)]{geaquart}, namely,
\begin{equation}\label{ramanujanomeganuz}
\nu(z;q)=\frac{(q^2;q^2)_{\infty}(-q^2;q^2)_{\infty}^2}{(z^2q^2;q^4)_{\infty}}-zq\omega(z^2;q^2).
\end{equation}
Furthermore, letting $z=1$ in \eqref{ramanujanomeganuz} gives Ramanujan's relation between $\omega(q)$ and $\nu(q)$ \cite[p.~62, Equation (26.88)]{fine}, that is,
\begin{equation}\label{ramanujanomeganu}
\nu(q)=(q^2;q^2)_{\infty}(-q^2;q^2)_{\infty}^3-q\omega(q^2).
\end{equation}
It should be noted that \eqref{ramanujanomeganu} is crucially used in \cite[p.~6]{ady1} in establishing \eqref{pwn}. Moreover, in our direct proof of \eqref{pwnz}, we need \eqref{ramanujanomeganuz}. Hence, to realize our goal of generalizing \eqref{pwnz} and \eqref{pnnz}, it is natural to consider  $\omega_1(\a, z;q)$ instead of $\omega_0(\a, z;q)$. Also, the other trivariate generalization of $\nu(q)$ that we study, namely, $\nu(\a, z; q)$, is the same as that considered by Li and Yang \cite{liyangmock} and also by Choi \cite{choi}, in view of \eqref{equivalence}, since
\begin{equation*}
\nu(\a, z; q)=\nu_{0}(-z,-\a/z;q).
\end{equation*}
We are now ready to give generalizations of the Andrews-Yee identities \eqref{pwnz} and \eqref{pnnz}. We begin with that of \eqref{pwnz}.
\begin{theorem}\label{theorem3}
Let $\omega_1(\alpha,z;q)$ be defined by \eqref{2b}.  Then
\begin{align}\label{result1}
\sum_{m=1}^{\infty}&\frac{(z^2q^2/\a;q^2)_{m-1}q^{2m}}{(-zq/\a;q)_{2m}(zq;q)_{2m}(-zq^{2m+2}/\a;q^2)_\infty(zq^{2m+2};q^2)_\infty}\nonumber\\
=&\frac{q^2(z^2q^2/\a;q^2)_\infty}{(zq;q)_\infty(-zq/\a;q)_\infty}\omega_{1}(\a,z;q)+F(a,z;q),
\end{align}
where
\begin{align}\label{F}
F(\alpha,z;q):=&
\frac{1}{(-zq^2/ \a;q^2)_\infty(zq^2;q^2)_\infty}\Bigg(\frac{q^2}{(1+zq/\a)}
\sum_{m =0}^{\infty}\frac{(-\a/(zq);q^2)_{m+1}}{(-zq;q^2)_{m+1}}(zq)^m\notag\\
&-\frac{\a q}{z}\frac{(z^2q^2/\a;q^2)_\infty}{(z^2q^2;q^4)_{\infty}(z^2q^2/\a^2;q^{4})_{\infty}}\sum_{n=0}^{\infty}\a^nq^{n^2+n}\Bigg).
\end{align}
\end{theorem}
On the other hand, the generalization of \eqref{pnnz} that we obtain is now given.

\begin{theorem}\label{theorem4}
Let $\nu(\a, z; q)$ be defined in \eqref{1}. Then,
\begin{align}\label{gsayeqn1}
\sum_{n=0}^{\i}\frac{(-zq;q)_{2n}(-zq^{2n+2};q^2)_\i}{\left( -zq/\a;q\right)_{n}}q^n   =& -\frac{\a}{z}\frac{(-zq;q)_{\i}}{\left( -zq/\a;q\right)_{\i}}\nu\left(\frac{\a^2}{z},-\frac{\a^2}{z};-q\right)\nonumber\\
&+(-zq^2;q^2)_{\infty}\sum_{n=0}^{\infty}\frac{\left(-\a/z;q\right)_{n+1}\a^{n}q^{n(n+1)/2}}{\left(-\a^2 q/z;q^2\right)_{n+1}}.
\end{align}
\end{theorem}

There is a trade-off between the two generalizations of the Andrews-Yee identities above.  Although Theorem \ref{theorem3} appears somewhat complicated, the variables $\a, z$, and $q$ occurring in $\omega_{1}(\a, z; q)$ are independent. On the other hand, Theorem \ref{theorem4} is comparatively simpler than Theorem \ref{theorem3}, but the first two components in $\nu\left(\a^2/z,-\a^2/z;-q\right)$ are dependent.

A proof of Theorem \ref{theorem3} is given in Section \ref{***} along with the derivation of \eqref{pwnz} as a special case of Theorem \ref{theorem3}. We also give an important application of this theorem in the same section. Also, apart from proving Theorem \ref{theorem4} in Section \ref{aysecond}, we show how \eqref{pnnz} can be derived from Theorem \ref{theorem4}. Numerous corollaries of Theorem \ref{theorem4}, which give many well-known results in the literature as well as new ones, are given in Section \ref{aysecond1}.

Throughout this paper, it is assumed that the variables $q, \alpha$, and $z$ are chosen so that all functions of one or more of these variables are analytic.

\section{Preliminary results}\label{prelim}
For $z\neq0$ and $|q|<1$, the Jacobi triple product identity is given by \cite[p.~21, Theorem 2.8]{gea}
\begin{equation}\label{jtpi}
\sum_{n=-\infty}^{\infty}z^nq^{n^2}=(-zq;q^2)_{\infty}(-q/z;q^2)_{\infty}(-q^2;q^2)_{\infty}.
\end{equation}
We will use a deep identity of Andrews \cite[p.~141, Theorem 1]{gea90} in our proofs of Theorems \ref{theorem3} and \ref{theorem4}. For $a,A\neq0$,
\begin{align}
 \sum_{n=0}^{\infty} \frac{(B;q)_n (-Abq;q)_n q^n}{(-aq;q)_n (-bq;q)_n}
&=-\frac{(B;q)_{\infty} (-Abq;q)_{\infty}}{a(-bq;q)_{\infty} (-aq;q)_{\infty}} \sum_{m=0}^{\infty}
\frac{(A^{-1};q)_m \left(Abq/a\right)^m}{\left(-B/a;q\right)_{m+1}}\nonumber\\
&\quad+(1+b) \sum_{m=0}^{\infty} \frac{(-a^{-1};q)_{m+1}( -ABq/a;q)_{m} (-b)^m}{\left(-B/a;q\right)_{m+1} \left(Abq/a;q\right)_{m+1}}. \label{gea90_thm1}
\end{align}
We also require the $q$-analogue of Gauss's second theorem, namely \cite[Equation (1.8)]{qana},
\begin{align}\label{gauss}
\sum_{n=0}^{\infty}\frac{(a;q)_{n}(b;q)_n}{(q;q)_n(qab ;q^2)_{n}}q^{n(n+1)/2} = \frac{(-q;q)_{\infty}(aq;q^2)_{\infty}(bq;q^2)_{\infty}}{(qab;q^2)_{\infty}}.
\end{align}

\section{A functional relation between $\omega_1(\a, z;q)$ and $\nu(\a, z;q)$}\label{proof}
The functional relation in Theorem \ref{newrelation} is proved in this section.

\begin{proof}[Theorem \textup{\ref{newrelation}}][]
 We begin by quoting a corrected version of a formula from Fine's book \cite[Equation~(8.2)]{fine}.  (In (8.1), for $H$, replace $(ut)_{\i}$ by $(-ut)_{\i}$, and replace  $(u)_{\i}$ by $(-u)_{\i}$; for $G$, replace $(ut)_{n}$ by $(-ut)_{n}$. In (8.2), replace  $(u)_{\i}$ by $(-u)_{\i}$.)  To that end,
\begin{equation}\label{nfine}
\sum_{n=0}^{\infty}\df{(b/u)^nq^{(n^2+n)/2}}{(bq)_n} -\df{u}{1+u}\sum_{n=0}^{\infty}\df{q^n}{(bq)_n(-uq)_n}
=\df{1}{(bq)_{\infty}(-u)_{\infty}}\sum_{n=0}^{\infty}(b/u)^nq^{(n^2+n)/2}.
\end{equation}
Replacing $q$ by $q^2$ in \eqref{nfine}, we deduce that
\begin{gather*}
\frac{1}{(1-b)}\sum_{n=0}^{\infty}\frac{(b/u)^nq^{n^2+n}}{(bq^2;q^2)_n}\\
=\dfrac{1}{(b;q^2)_{\infty}(-u;q^2)_\infty}\sum_{n=0}^{\infty}(b/u)^nq^{(n^2+n)}
+\frac{u}{(1-b)(1+u)}\sum_{n=0}^{\infty}\frac{q^{2n}}{(bq^2;q^2)_n(-uq^2;q^2)_n}.
\end{gather*}
Taking $u=-zq/\a$ and $ b=-zq$, we furthermore find that
{\allowdisplaybreaks\begin{align*}
\sum_{n=0}^{\infty}\frac{\a^nq^{n^2+n}}{(-zq;q^2)_{n+1}}
=&\frac{1}{(-zq;q^2)_{\infty}(zq/\a;q^2)_\infty}\sum_{n=0}^{\infty}\a^nq^{n^2+n}\\
&-\frac{zq}{\a}\frac{1}{(1+zq)(1-zq/\a)}\sum_{n=0}^{\infty}\frac{q^{2n}}{(-zq^3;q^2)_n(zq^3/\a;q^2)_n}\\
=&\frac{1}{(-zq;q^2)_{\infty}(zq/\a;q^2)_\infty}\sum_{n=0}^{\infty}\a^nq^{n^2+n}\\
&-\frac{zq}{\a}\sum_{n=0}^{\infty}\frac{q^{2n}}{(-zq;q^2)_{n+1}(zq/\a;q^2)_{n+1}}.
\end{align*}}
Using \eqref{1} and \eqref{2b}, we complete the proof.
\end{proof}

Fine's identity  \eqref{nfine} can be found in Ramanujan's Lost Notebook \cite[p.~40]{lnb}, \cite[Entry 6.3.1, p.~115]{AB2}.

\section{Generalization of the first Andrews-Yee identity and its application}\label{***}
Here, we first prove Theorem \ref{theorem3}. Then we derive \eqref{pwnz} from it. Finally, we give a striking application of Theorem \ref{theorem3}.

\begin{proof}[Theorem \textup{\ref{theorem3}}][]

We employ \eqref{gea90_thm1}.  First, replace $q$ by $q^2$.  Second, let  $ B=z^2q^2/\a, a=zq/\a$, and $b=-zq$. Third, let $A=0$ while using the limit
$$ \lim_{A\to0}(1/A;q^2)_{m}A^m=\lim_{A\to0}(1-1/A)(1-q^2/A)\cdots(1-q^{2m-2}/A)A^m=(-1)^mq^{m(m-1)}.$$
We then deduce that
\begin{align*}
\sum_{m=0}^{\infty}\frac{(z^2q^2/\a;q^2)_m}{(-zq^3/\a;q^2)_m(zq^3;q^2)_m}q^{2m}=& -\frac{\a}{zq}\frac{(z^2q^2/\a;q^2)_\infty}{(zq^3;q^2)_\infty(-zq^3/\a;q^2)_\infty}\nu(\a,z;q)\\
&+(1-zq)\sum_{n=0}^{\infty}\frac{(-\a/(zq);q^2)_{m+1}}{(-zq;q^2)_{m+1}}(zq)^m.
\end{align*}
Now substitute $\nu(\a,z,q)$ from \eqref{newrelationeqn} into the last equality to deduce that
\begin{align}
&\sum_{m=0}^{\infty}\frac{(z^2q^2/\a;q^2)_m}{(-zq^3/\a;q^2)_m(zq^3;q^2)_m}q^{2m}\notag\\
=& -\frac{\a}{zq}\frac{(z^2q^2/\a;q^2)_\infty}{(zq^3;q^2)_\infty(-zq^3/\a;q^2)_\infty}
\left(\frac{1}{(-zq;q^2)_{\infty}(zq/\a;q^{2})_{\infty}}\sum_{n=0}^{\infty}\a^nq^{n^2+n} -\frac{zq}{\a}\omega_{1}(\a, z; q)\right)\notag\\
&\qquad \qquad+(1-zq)\sum_{n=0}^{\infty}\frac{(-\a/(zq);q^2)_{m+1}}{(-zq;q^2)_{m+1}}(zq)^m\notag\\
=&\frac{(z^2q^2/\a;q^2)_\infty}{(zq^3;q^2)_\infty(-zq^3/\a;q^2)_\infty}\omega_{1}(\a,z;q)\notag\\
&-\frac{\a}{zq}\frac{(z^2q^2/\a;q^2)_\infty
}{(zq^3;q^2)_\infty(-zq^3/\a;q^2)_\infty(-zq;q^2)_{\infty}(zq/\a;q^{2})_{\infty}}\sum_{n=0}^{\infty}\a^nq^{n^2+n}\notag\\
&\qquad \qquad+(1-zq)\sum_{n=0}^{\infty}\frac{(-\a/(zq);q^2)_{m+1}}{(-zq;q^2)_{m+1}}(zq)^m.\label{firsttag}
\end{align}
If we divide both sides above by $(1+zq/\a)(1-zq)(-zq^2/ \a;q^2)_\infty(zq^2;q^2)_\infty$,  the left-hand side becomes
{\allowdisplaybreaks\begin{align}
&\frac{1}{(1+zq/\a)(1-zq)(-zq^2/ \a;q^2)_\infty(zq^2;q^2)_\infty}
\sum_{m=0}^{\infty}\frac{(z^2q^2/\a;q^2)_m}{(-zq^3/\a;q^2)_m(zq^3;q^2)_m}q^{2m}\notag\\
&=\frac{1}{(-zq^2/ \a;q^2)_\infty(zq^2;q^2)_\infty}\sum_{m=0}^{\infty}
\frac{(z^2q^2/\a;q^2)_m(-zq^2/\a;q^2)_{m+1}(zq^2;q^2)_{m+1}}{(-zq/\a;q)_{2m+2}(zq;q)_{2m+2}}q^{2m}\notag\\
&=\sum_{m=0}^{\infty}\frac{(z^2q^2/\a;q^2)_mq^{2m}}{(-zq/\a;q)_{2m+2}(zq;q)_{2m+2}(-zq^{2m+4}/\a;q^2)_\infty(zq^{2m+4};q^2)_\infty}\notag\\
&=\frac{1}{q^2}\sum_{m=1}^{\infty}\frac{(z^2q^2/\a;q^2)_{m-1}q^{2m}}{(-zq/\a;q)_{2m}(zq;q)_{2m}(-zq^{2m+2}/\a;q^2)_\infty(zq^{2m+2};q^2)_\infty},
\label{secondtag}
\end{align}}
while the right-hand side becomes
\begin{align}
&\frac{1}{(1+zq/\a)(1-zq)(-zq^2/ \a;q^2)_\infty(zq^2;q^2)_\infty}\Bigg(\frac{(z^2q^2/\a;q^2)_\infty}{(zq^3;q^2)_\infty(-zq^3/\a;q^2)_\infty}\omega_{1}(\a,z;q)\notag\\
&-\frac{\a}{zq}\frac{(z^2q^2/\a;q^2)_\infty
}{(zq^3;q^2)_\infty(-zq^3/\a;q^2)_\infty(-zq;q^2)_{\infty}(zq/\a;q^{2})_{\infty}}\sum_{n=0}^{\infty}\a^nq^{n^2+n}\notag\\
&+(1-zq)
\sum_{m=0}^{\infty}\frac{(-\a/(zq);q^2)_{m+1}}{(-zq;q^2)_{m+1}}(zq)^m \Bigg)\notag \\
=&\frac{(z^2q^2/\a;q^2)_\infty}{(zq;q)_\infty(-zq/\a;q)_\infty}\omega_{1}(\a,z;q)\notag\\&-\frac{\a}{zq}\frac{(z^2q^2/\a;q^2)_\infty
}{(-zq/ \a;q)_\infty(zq;q)_\infty(-zq;q^2)_{\infty}(zq/\a;q^{2})_{\infty}}\sum_{n=0}^{\infty}\a^nq^{n^2+n}\notag\\
&+\frac{1}{(1+zq/\a)(-zq^2/ \a;q^2)_\infty(zq^2;q^2)_\infty}\sum_{m=0}^{\infty}\frac{(-\a/(zq);q^2)_{m+1}}{(-zq;q^2)_{m+1}}(zq)^m .\label{thirdtag}
\end{align}
Remembering that we had divided both sides of \eqref{firsttag} by $(1+zq/\a)(1-zq)(-zq^2/ \a;q^2)_\infty(zq^2;q^2)_\infty$, we substitute \eqref{secondtag} and \eqref{thirdtag} into \eqref{firsttag} as amended.  We then multiply both sides by the resulting identity by $q^2$ to conclude that
\begin{align*}
\sum_{m=1}^{\infty}&\frac{(z^2q^2/\a;q^2)_{m-1}q^{2m}}{(-zq/\a;q)_{2m}(zq;q)_{2m}(-zq^{2m+2}/\a;q^2)_\infty(zq^{2m+2};q^2)_\infty}\nonumber\\
=&\frac{(z^2q^2/\a;q^2)_\infty}{(zq;q)_\infty(-zq/\a;q)_\infty}q^2\omega_{1}(\a,z;q)+F(\alpha,z;q),
\end{align*}
where
{\allowdisplaybreaks\begin{align*}
F(\a,z;q)=&\frac{q^2}{(1+zq/\a)(-zq^2/ \a;q^2)_\infty(zq^2;q^2)_\infty}\sum_{m=0}^{\infty}\frac{(-\a/(zq);q^2)_{m+1}}{(-zq;q^2)_{m+1}}(zq)^m\\
&-\frac{\a q}{z}\frac{(z^2q^2/\a;q^2)_\infty}{(-zq/ \a;q)_\infty(zq;q)_\infty(-zq;q^2)_{\infty}(zq/\a;q^{2})_{\infty}}
\sum_{n=0}^{\infty}\a^nq^{n^2+n}\\
=&\frac{1}{(-zq^2/ \a;q^2)_\infty(zq^2;q^2)_\infty}\Bigg(\frac{q^2}{(1+zq/\a)}\sum_{m =0}^{\infty}\frac{(-\a/(zq);q^2)_{m+1}}{(-zq;q^2)_{m+1}}(zq)^m\\
&-\frac{\a q}{z}\frac{(z^2q^2/\a;q^2)_\infty}{(z^2q^2;q^4)_{\infty}(z^2q^2/\a^2;q^{4})_{\infty}}\sum_{n=0}^{\infty}\a^nq^{n^2+n}\Bigg).
\end{align*}}
This completes our proof.
\end{proof}
\subsection{Andrews and Yee's \eqref{pwnz} as a special case of Theorem \ref{theorem3}}\label{ay1spl} We begin with a lemma.

\begin{lemma}\label{1psi1spl}
For $|zq|<1$ and $|q|<1$,
\begin{equation}\label{1psi1spleqn}
\sum_{m=0}^{\infty}\frac{(-q/z;q^2)_m}{(-zq^3;q^2)_m}(zq)^m
=\frac{(1+zq)(z^2q^4;q^4)_{\infty}(q^2;q^2)_{\infty}(-q^2;q^2)_{\infty}^{2}}{(z^2q^2;q^4)_{\infty}}.
\end{equation}
\end{lemma}

\begin{proof}
 Let $z=y/q$ in \eqref{1psi1spleqn}. Thus,
\begin{equation}\label{2a}
\sum_{m=0}^{\i}\df{(-q^2/y;q^2)_m}{(-yq^2;q^2)_m}y^m=\df{(1+y)(y^2q^2;q^4)_{\i}(q^2;q^2)_{\i}(-q^2;q^2)_{\i}^2}{(y^2;q^4)_{\i}}.
\end{equation}
Replace $q^2$ by $q$ in \eqref{2a} to find that
\begin{equation}\label{3a}
\sum_{m=0}^{\i}\df{(-q/y;q)_m}{(-yq;q)_m}y^m=\df{(1+y)(y^2q;q^2)_{\i}(q;q)_{\i}(-q;q)_{\i}^2}{(y^2;q^2)_{\i}}.
\end{equation}
Our goal is thus to prove \eqref{3a}.

Recall next the definition of the basic hypergeometric function ${_2\phi_1}$,
\begin{equation*}
{_2\phi_1}(a,b;c;q,z):=\sum_{n=0}^{\infty}\df{(a;q)_n(b;q)_n}{(c;q)_n(q;q)_n}z^n,\qquad |z|<1,
\end{equation*}
and
 Heine's transformation \cite[p.~13]{gasper}
\begin{equation*}
_{2}\phi_{1}(a,b;c;q,z)=\df{(b;q)_{\i}(az;q)_{\i}}{(c;q)_{\i}(z;q)_{\i}} {_{2}\phi_{1}}(c/b,z;az;q,b).
\end{equation*}
Set $b=q, c=-yq,a=-q/y$, and $z=y$ to find that
\begin{align}\label{4a}
\sum_{m=0}^{\i}\df{(-q/y;q)_m}{(-yq;q)_m}y^m
&=\df{(q;q)_{\i}(-q;q)_{\i}}{(-yq;q)_{\i}(y;q)_{\i}}\sum_{m=0}^{\i}\df{(-y;q)_m(y;q)_m}{(-q;q)_m(q;q)_m}q^m\notag\\
&=\df{(1+y)(q^2;q^2)_{\i}}{(y^2;q^2)_{\i}}\sum_{m=0}^{\i}\df{(y^2;q^2)_m}{(q^2;q^2)_m}q^m.
\end{align}
When we compare \eqref{4a} with \eqref{3a}, we see that we must show that
\begin{equation}\label{5a}
 \sum_{m=0}^{\i}\df{(y^2;q^2)_m}{(q^2;q^2)_m}q^m=(y^2q;q^2)_{\i}(-q;q)_{\i}=\df{(y^2q;q^2)_{\i}}{(q;q^2)_{\i}}.
 \end{equation}
 To prove \eqref{5a}, we apply the $q$-binomial theorem \cite[p.~8]{gasper}
 \begin{equation}\label{6a}
 \sum_{m=0}^{\i}\df{(a;q)_m}{(q;q)_m}z^m=\df{(az;q)_{\i}}{(z;q)_{\i}}.
 \end{equation}
 Setting $a=y^2$ and $z=q$ and replacing $q$ by $q^2$ in \eqref{6a}, we arrive at
 \begin{equation}\label{7a}
 \sum_{m=0}^{\i}\df{(y^2;q^2)_m}{(q^2;q^2)_m}q^m=\df{(y^2q;q^2)_{\i}}{(q;q^2)_{\i}}.
 \end{equation}
 Noting that \eqref{5a} and \eqref{7a} are identical, we complete the proof.
\end{proof}

We next give a second proof of Lemma \ref{1psi1spl}.

\begin{proof}
Replacing $q$ by $q^2$ in Ramanujan's ${}_1\psi_{1}$ summation \cite[p.~139, Equation (5.2.2)]{gasper}, we have, for $|b/a|<|t|<1$ and $|q|<1$,
\begin{equation}\label{1psi1}
\sum_{m=0}^{\infty}\frac{(a;q^2)_m}{(b;q^2)_m}t^m
+\sum_{m=1}^{\infty}\frac{(q^2/b;q^2)_m}{(q^2/a;q^2)_m}\left(\frac{b}{at}\right)^m
=\frac{(q^2;q^2)_{\infty}(b/a;q^2)_{\infty}(at;q^2)_{\infty}(q^2/(at);q^2)_{\infty}}
{(b;q^2)_{\infty}(q^2/a;q^2)_{\infty}(t;q^2)_{\infty}(b/(at);q^2)_{\infty}}.
\end{equation}
Let $a=-q/z, b=-zq^3$ and $t=zq$ above and thereby note that both series on the left side are equal. This gives \eqref{1psi1spleqn} upon simplification.
\end{proof}

\begin{proof}[\textup{\eqref{pwnz}}][]
Let $\a=1$ in Theorem \ref{theorem3}. By Lemma \ref{1psi1spl}, $F(1,z,q)=0$.  Thus, putting $\alpha=1$ in \eqref{result1} and using \eqref{omega1omega}, we see that \eqref{result1} reduces to \eqref{pwnz}.
\end{proof}

\subsection{An application of Theorem \ref{theorem3}}\label{ay1app}
In this subsection we show that the following corollary is equivalent to a special case of a beautiful reciprocity theorem of S.-Y.~Kang \cite{kang}.
\begin{corollary}\label{frelcor}
For $F(\alpha,z;q)$ defined by \eqref{F},
\begin{equation}\label{frel}
F(\a, z;q)=F\left(\frac{1}{\a},-\frac{z}{\a};q\right).
\end{equation}
\end{corollary}
\begin{proof}
Define
\begin{align}\label{defg}
G(\a, z;q):=\sum_{m=1}^{\infty}&\frac{(z^2q^2/\a;q^2)_{m-1}q^{2m}}{(-zq/\a;q)_{2m}(zq;q)_{2m}(-zq^{2m+2}/\a;q^2)_\infty(zq^{2m+2};q^2)_\infty},\end{align}
and
\begin{align*}
H(\a, z;q)&:=\frac{(z^2q^2/\a;q^2)_\infty}{(zq;q)_\infty(-zq/\a;q)_\infty}q^2\omega_{1}(\a,z;q)\nonumber\\
&=\frac{q^2(z^2q^2/\a;q^2)_\infty}{(zq;q)_\infty(-zq/\a;q)_\infty}\sum_{n=0}^{\infty}\frac{q^{2n}}{(-zq;q^2)_{n+1}(zq/\a;q^{2})_{n+1}},
\end{align*}
where $\omega_{1}(\a, z; q)$ is defined in \eqref{2b}. It is straightforward to see that
\begin{align}\label{relgh}
G(\a, z;q)&=G\left(\frac{1}{\a},-\frac{z}{\a};q\right),\nonumber\\
H(\a, z;q)&=H\left(\frac{1}{\a},-\frac{z}{\a};q\right).
\end{align}
Thus, \eqref{relgh} along with Theorem \ref{theorem3} implies \eqref{frel}.
\end{proof}

One of the beautiful theorems in Ramanujan's lost notebook is his reciprocity theorem \cite[p.~40]{lnb}, \cite{AB2}.
\begin{theorem}\label{ramanujanrec}
 For $ab\neq0$,
\begin{gather}
\left(1+\frac{1}{b}\right)\sum_{n=0}^{\infty}\frac{(-1)^nq^{n(n+1)/2}a^nb^{-n}}{(-aq)_{n}}
-\left(1+\frac{1}{a}\right)\sum_{n=0}^{\infty}\frac{(-1)^nq^{n(n+1)/2}a^{-n}b^{n}}{(-bq)_{n}}\notag\\
=\left(\frac{1}{b}-\frac{1}{a}\right)\frac{(aq/b, bq/a, q;q)_{\infty}}{( -aq, -bq;q)_{\infty}}.\label{reciprocitytheorem}
\end{gather}
\end{theorem}

Soon-Yi Kang obtained the following beautiful generalization of Ramanujan's reciprocity theorem \cite[Theorem 4.1]{kang}.

\begin{theorem}\label{kangthm} If $|c|<|a|<1$ and $|c|<|b|<1$, then
\begin{align}\label{3varrt}
\rho_{3}(a, b, c;q)-\rho_{3}(b, a, c;q)=\left(\frac{1}{b}-\frac{1}{a}\right)\frac{(c,aq/b, bq/a, q;q)_{\infty}}{(-c/a, -c/b, -aq, -bq;q)_{\infty}},
\end{align}
where
\begin{align}\label{rho3}
\rho_{3}(a, b, c;q):=\left(1+\frac{1}{b}\right)\sum_{n=0}^{\infty}\frac{(c)_n(-1)^nq^{n(n+1)/2}a^nb^{-n}}{(-aq)_{n}(-c/b)_{n+1}}.
\end{align}
\end{theorem}

Kang \cite[p.~24]{kang} shows that \eqref{3varrt} is equivalent to the ${}_1\psi_{1}$ summation formula \eqref{1psi1}. Also, Ramanujan's reciprocity theorem \eqref{reciprocitytheorem} is the special case $c=0$ of Theorem \ref{kangthm}.
It turns out that a special case of Kang's result \eqref{3varrt} is simply a restatement of Corollary \ref{frelcor}. This is given in the following corollary.

\begin{corollary}
The identity \eqref{frel} is equivalent to the special case $a=zq$, $b=-zq/\a$ and $c=z^2q^2/\a$ of Kang's result \eqref{3varrt}.
\end{corollary}

\begin{proof}
Substituting the expressions for $F(\a, z;q)$ and $F\left(1/\alpha,-z/\alpha;q\right)$ from \eqref{F} in \eqref{frel} and simplifying, we see that
\begin{align}\label{equating}
&\frac{q^2}{(1+zq/\a)}\sum_{m =0}^{\infty}\frac{(-\a/(zq);q^2)_{m+1}}{(-zq;q^2)_{m+1}}(zq)^m
-\frac{q}{z}\frac{(z^2q^2/\a;q^2)_\infty}{(z^2q^2;q^4)_{\infty}(z^2q^2/\a^2;q^{4})_{\infty}}\sum_{n=0}^{\infty}\a^{n+1}q^{n^2+n}\nonumber\\
&=\frac{q^2}{(1-qz)}\sum_{m =0}^{\infty}\frac{(1/(zq);q^2)_{m+1}}{(zq/\a;q^2)_{m+1}}\left(-\frac{zq}{\a}\right)^m
+\frac{q}{z}\frac{(z^2q^2/\a;q^2)_\infty}{(z^2q^2;q^4)_{\infty}(z^2q^2/\a^2;q^{4})_{\infty}}\sum_{n=0}^{\infty}\a^{-n}q^{n^2+n}.
\end{align}
Replace $n$ by $-n-1$ in the partial theta function on the left-hand side.  Thus,
\begin{equation}\label{cov}
\sum_{n=0}^{\infty}\a^{n+1}q^{n^2+n}=\sum_{n=-\infty}^{-1}\a^{-n}q^{n^2+n}.
\end{equation}
Hence, from \eqref{equating} and \eqref{cov}, we see, upon simplification, that
\begin{align}\label{app0}
&\frac{(1+\a/(zq))}{(1+zq/\a)(1+zq)}\sum_{m =0}^{\infty}\frac{(-\a q/z;q^2)_{m}}{(-zq^3;q^2)_{m}}(zq)^m-\frac{(1-1/(zq))}{(1-qz)(1-zq/\a)}\sum_{m =0}^{\infty}\frac{(q/z;q^2)_{m}}{(zq^3/\a;q^2)_{m}}\left(-\frac{zq}{\a}\right)^m\nonumber\\
&=\frac{1}{zq}\frac{(z^2q^2/\a;q^2)_\infty}{(z^2q^2;q^4)_{\infty}(z^2q^2/\a^2;q^{4})_{\infty}}\sum_{n=-\infty}^{\infty}\left(\frac{\a}{q}\right)^{n}q^{n^2}.
\end{align}
The Jacobi triple product identity \eqref{jtpi} gives
\begin{align}\label{jac}
\sum_{n=-\infty}^{\infty}\left(\frac{\a}{q}\right)^{n}q^{n^2}=\left(-\a,-\frac{q^2}{\a},q^2;q^2\right)_{\infty}.
\end{align}
Moreover, the $q$-analogue of Pfaff's transformation established by Jackson \cite[p.~527]{qana} is given by
\begin{align}\label{kum}
\sum_{n=0}^{\infty}\frac{(c/b;q)_n(a;q)_nx^n}{(q;q)_n(c;q)_n}=\frac{(ax;q)_\infty}{(x;q)_\infty}
\sum_{n=0}^{\infty}\frac{(a;q)_n(b;q)_n(-1)^nq^{n(n-1)/2}(x c/b)^n}{(q;q)_n(c;q)_n(ax;q)_n}.
\end{align}
Replace $q$ by $q^2$, then let $a=q^2, c=-zq^3, b=z^2q^2/\a$ and $x=zq$ in \eqref{kum} to find that
\begin{align}\label{app1}
\sum_{m =0}^{\infty}\frac{(-\a q/z;q^2)_{m}}{(-zq^3;q^2)_{m}}(zq)^m&=\sum_{n=0}^{\infty}\frac{(z^2q^2/\a;q^2)_n\a^n q^{n^2+n}}{(-zq^3;q^2)_n(zq;q^2)_{n+1}}\nonumber\\
&=\frac{1}{(1-\a/(zq))}\rho_3\left(zq,-\frac{zq}{\a},\frac{z^2q^2}{\a};q^2\right),
\end{align}
where $\rho_{3}(a, b, c)$ is defined in \eqref{rho3}. Similarly, replacing $q$ by $q^2$, then letting $a=q^2,c=zq^3/\a, b=z^2q^2/\a$ and $x=-zq/\a$ in \eqref{kum}, we are led to
\begin{align}\label{app2}
\sum_{m =0}^{\infty}\frac{(q/z;q^2)_{m}}{(zq^3/\a;q^2)_{m}}\left(-\frac{zq}{\a}\right)^m
&=\sum_{n=0}^{\infty}\frac{(z^2q^2/\a;q^2)_n\a^{-n}q^{n^2+n}}{(zq^3/\a;q^2)_n(-zq/\a;q^2)_{n+1}}\nonumber\\
&=\frac{1}{(1+1/(zq))}\rho_3\left(-\frac{zq}{\a}, zq, \frac{z^2q^2}{\a};q^2\right).
\end{align}
Now substitute \eqref{jac}, \eqref{app1}, and \eqref{app2} in \eqref{app0} and simplify to deduce that
\begin{align}\label{3varrt1}
&-\frac{1}{(1+zq)(1-zq/\a)}\left\{\rho_3\left(zq,-\frac{zq}{\a},\frac{z^2q^2}{\a};q^2\right)-\rho_3\left(-\frac{zq}{\a}, zq, \frac{z^2q^2}{\a};q^2\right)\right\}\nonumber\\
&=\frac{1}{zq}\frac{(z^2q^2/\a, -\a, -q^2/\a, q^2;q^2)_\infty}{(z^2q^2;q^4)_{\infty}(z^2q^2/\a^2;q^{4})_{\infty}}.
\end{align}
Multiplying both sides of \eqref{3varrt1} by $-(1+zq)(1-zq/\a)$ and simplifying leads us to \eqref{3varrt} with $a=zq$, $b=-zq/\a$ and $c=z^2q^2/\a$. Since the steps are clearly reversible, we see that \eqref{frel} can also be derived from this special case of \eqref{3varrt}.
\end{proof}

\section{Generalization of the second Andrews-Yee identity}\label{aysecond}

In this section, we will first prove Theorem \ref{theorem4} and then derive \eqref{pnnz} as a special case of Theorem \ref{theorem4}.
\begin{proof}[Theorem \textup{\ref{theorem4}}][]
Note that the left-hand side can be written in the form
\begin{align}\label{simplification}
\sum_{n=0}^{\i}\frac{(-zq;q)_{2n}(-zq^{2n+2};q^2)_\i}{\left( -zq/\a;q\right)_{n}}q^n=(-zq^2;q^2)_\i\sum_{n=0}^{\infty} \frac{(-zq;q^2)_n }{(-zq/\a;q)_n }q^n.
\end{align}
We would like to transform the sum on the right-hand side of \eqref{simplification}.
To that end, we apply \eqref{gea90_thm1}
with $ A=B/bq$, $a= z/\a$ and $b \to 0$ to find that
\begin{align}\label{afterandrews}
 \sum_{n=0}^{\infty} \frac{(B^2;q^2)_n }{(-zq/\a;q)_n }q^n
&=-\frac{\a(B^2;q^2)_\i }{z(-zq/\a;q)_\i } \sum_{m=0}^{\infty} \frac{\left(B\a/z\right)^m}{\left(-B\a/z;q\right)_{m+1}}\nonumber\\
&\quad+\sum_{m=0}^{\infty} \frac{(-\a/z;q)_{m+1}}{\left(B^2\a^2 /z^2;q^2\right)_{m+1}}\left(-B^2\a / z\right)^mq^{m(m-1)/2}.
\end{align}
From \cite[Equation~(41)]{ady1},
\begin{align}\label{eq41}
\sum_{m=0}^{\infty}\frac{(B/z)^m}{(-B/z ;q)_{m+1}}=\sum_{m=0}^{\infty}(q;q^2)_m\left(-B/z\right)^{2m}.
\end{align}
Replace $z$ by $z/\a$ in \eqref{eq41} and substitute the resulting identity
$$\sum_{m=0}^{\infty} \frac{\left(B\a/z\right)^m}{\left(-B\a/z;q\right)_{m+1}}$$
on the right-hand side of \eqref{afterandrews}, so as to obtain
\begin{align}\label{BBB}
 \sum_{n=0}^{\infty} \frac{(B^2;q^2)_n }{(-zq/\a;q)_n }q^n
=&-\frac{\a(B^2;q^2)_\i }{z(-zq/\a;q)_\i } \sum_{m=0}^{\i}(q;q^2)_m\left(-B\a / z\right)^{2m}\nonumber\\
&+\sum_{m=0}^{\infty} \frac{(-\a/z;q)_{m+1}}{\left(B^2\a^2/z^2;q^2\right)_{m+1}}\left(-B^2\a / z\right)^{m}q^{m(m-1)/2}.
\end{align}
Now let $B^2 = -z q$ in \eqref{BBB} to deduce that
\begin{align}\label{aftersubs}
 \sum_{n=0}^{\infty} \frac{(-zq;q^2)_n }{(-zq/\a;q)_n }q^n
=&-\frac{\a(-zq;q^2)_\i }{z(-zq/\a;q)_\i } \sum_{m=0}^{\i}(q;q^2)_m\left(-q\a^2/z\right)^{m}\nonumber\\
&+\sum_{m=0}^{\infty} \frac{(-\a/z;q)_{m+1}}{\left(-\a^2 q/z;q^2\right)_{m+1}}\a^{m}q^{m(m+1)/2}.
\end{align}

From \cite[p.~29, Exercise 6]{gea},
\begin{equation*}
\sum_{m=0}^{\infty}(-xq/y;q^2)_{m}y^m=\sum_{m=0}^{\i}\frac{q^{m^2}x^m}{(y;q^2)_{m+1}}.
\end{equation*}
Let $x=-y=\a^2q/z$ to deduce that
\begin{align}\label{andrewsex6}
\sum_{m=0}^{\i}(q;q^2)_m\left(-q\a^2/z\right)^{m}&=\sum_{m=0}^{\infty}\frac{q^{m^2+m}(\a^2/z)^m}{(-\a^2q/z;q^2)_{m+1}}\nonumber\\
&=\nu\left(\frac{\a^2}{z},-\frac{\a^2}{z};-q\right).
\end{align}
Finally, substitute \eqref{andrewsex6} in \eqref{aftersubs} and then use the resulting identity in \eqref{simplification} to arrive at \eqref{gsayeqn1}.
\end{proof}

As a corollary of Theorem \ref{theorem4}, we obtain \eqref{pnnz} as a special case. To derive this, however, we require the following lemma.
\begin{lemma}
Let $\nu(\a, z; q)$ be defined in \eqref{1}. Then
\begin{equation}\label{fromreciprocity}
\nu(\a,z;q)+\frac{1}{\a}\nu\left(\frac{1}{\a},-\frac{z}{\a};q\right)=\frac{(-\a q^2,-1/\a,q^2;q^2)_{\infty}}{(-zq,zq/\a;q^2)_{\infty}}.
\end{equation}
\end{lemma}

\begin{proof} Replace $q$ by $q^2$ in Ramanujan's reciprocity theorem \eqref{reciprocitytheorem}
and then set $a=zq$ and $b=-zq/\alpha$ in the resulting identity to deduce that
\begin{gather}
\left(1-\df{\alpha}{zq}\right)\sum_{n=0}^{\infty}\df{\alpha^nq^{n(n+1)}}{(-zq^3;q^2)_n}
-\left(1+\df{1}{zq}\right)\sum_{n=0}^{\infty}\df{\alpha^{-n}q^{n(n+1)}}{(zq^3/\alpha;q^2)_n}\notag\\
=\left(-\df{\alpha}{zq}-\df{1}{zq}\right)_{\infty}\df{(-\alpha q^2,-q^2/\alpha,q^2;q^2)_{\infty}}{(-zq^3,zq^3/\alpha;q^2)_{\infty}}.\label{rr1}
\end{gather}
Dividing both sides of \eqref{rr1} by $(1+zq)(1-\alpha/(zq))$ and simplifying, we find that
\begin{gather}
\sum_{n=0}^{\infty}\df{\alpha^nq^{n(n+1)}}{(-zq;q^2)_{n+1}}
+\df{1}{\alpha}\sum_{n=0}^{\infty}\df{\alpha^{-n}q^{n(n+1)}}{(zq/\alpha;q^2)_{n+1}}\notag\\
=-\df{(\alpha+1)}{zq}\df{(-\alpha q^2,-q^2/\alpha,q^2;q^2)_{\infty}}{(1-\alpha/(zq))(-zq,zq^3/\alpha;q^2)_{\infty}}.\label{rr2}
\end{gather}
Now,
\begin{align}\label{rr3}
\df{(-q^2/\alpha;q^2)_{\infty}}{(zq^3/\alpha;q^2)_{\infty}(1-\alpha/(zq))}
&=\df{(-1/\alpha;q^2)_{\infty}}{(zq^3/\alpha;q^2)_{\infty}(1+1/\alpha)(1-\alpha/(zq))}\notag\\
&=-\df{zq}{(\alpha+1)}\df{(-1/\alpha;q^2)_{\infty}}{(zq/\alpha;q^2)_{\infty}}.
\end{align}
Lastly, put \eqref{rr3} into \eqref{rr2} and recall the definition of $\nu(\alpha,z;q)$ from \eqref{1}.  We thus deduce that
$$ \nu(\a,z;q)+\frac{1}{\a}\nu\left(\frac{1}{\a},-\frac{z}{\a};q\right)=\frac{(-\a q^2,-1/\a,q^2;q^2)_{\infty}}{(-zq,zq/\a;q^2)_{\infty}},
$$
which is what we wanted to prove.

\end{proof}

\begin{corollary}\label{gsayspl}
For $|q|<1$ and $z\in \mathbb{Z}$,
\begin{align}\label{newprrof1}
\sum_{n=0}^{\infty}q^n(-zq^{n+1};q)_{n} (-zq^{2n+2};q^2)_{\infty}
=\nu_1(z;-q)=\sum_{n=0}^{\i}\frac{z^nq^{n^2+n}}{(q;q^2)_{n+1}}.
\end{align}
\end{corollary}

\begin{proof}
Let $\a=1$ in Theorem \ref{theorem3} to deduce that
\begin{align}\label{spad}
\sum_{n=0}^{\infty}q^n(-zq^{n+1};q)_{n} (-zq^{2n+2};q^2)_{\infty}=&-\frac{1}{z}\nu\left(\frac{1}{z},-\frac{1}{z};-q\right)\nonumber\\
&+(-zq^2;q^2)_{\infty}\sum_{n=0}^{\infty}\frac{\left(-1/z;q\right)_{n+1}q^{n(n+1)/2}}{\left(-q/z;q^2\right)_{n+1}}.
\end{align}
We now transform $\nu\left(1/z,-1/z;-q\right)$ into $\nu(z,1;-q)$ using \eqref{fromreciprocity}. To that end, replace $z$ and $q$ in \eqref{fromreciprocity} by $-1/z$ and $-q$, respectively, and then let $\a=1/z$, thereby obtaining
\begin{equation}\label{spab}
\nu\left(\dfrac{1}{z},-\dfrac{1}{z};-q\right)=\frac{(-z,-q^2/z,q^2;q^2)_{\infty}}{(q,-q/z;q^2)_{\infty}}-z\nu_1\left(z;-q\right),
\end{equation}
where we also used the fact that $\nu\left(z,1;-q\right)=\nu_1(z;-q)$, as can be seen from \eqref{nuqz2} and \eqref{1}.
Also, using  \eqref{gauss} with $a=q$ and $b=-q/z$ in the second equality below, and using  $(-q;q)_{\i}=1/(q;q^2)_{\i}$ in the third equality below, we find that
\begin{align}\label{spac}
\sum_{n=0}^{\infty}\frac{\left(-1/z;q\right)_{n+1}q^{n(n+1)/2}}{\left(-q/z;q^2\right)_{n+1}}
&=\frac{\left(1+1/z\right)}{\left(1+q/z\right)}
\sum_{n=0}^{\infty}\frac{\left(-q/z;q\right)_{n}q^{n(n+1)/2}}{\left(-q^3/z;q^2\right)_{n}}\notag\\
&=\frac{\left(1+1/z\right)}{\left(1+q/z\right)}\left\{\frac{(-q;q)_{\infty}(q^2,-q^2/z;q^2)_{\infty}}{(-q^3/z;q^2)_{\infty}}\right\}\notag\\
&=\df{(q^2,-1/z;q^2)_{\infty}}{(q,-q/z;q^2)_{\infty}}.
\end{align}
Hence, putting \eqref{spab} and \eqref{spac} in the right-hand side of \eqref{spad}, we arrive at
\begin{align}\label{rr5}
&-\frac{1}{z}\nu\left(\frac{1}{z},-\frac{1}{z};-q\right)
+(-zq^2;q^2)_{\infty}\sum_{n=0}^{\infty}\frac{\left(-1/z;q\right)_{n+1}q^{n(n+1)/2}}{\left(-q/z;q^2\right)_{n+1}}\notag\\
&=\nu_1(z;-q)- \frac{z(-z,-q^2/z,q^2;q^2)_{\infty}}{(q,-q/z;q^2)_{\infty}}
+(-zq^2;q^2)_{\infty}\df{(q^2,-1/z;q^2)_{\infty}}{(q,-q/z;q^2)_{\infty}}.
\end{align}
Now,
\begin{equation}\label{rr6}
(-zq^2;q^2)_{\infty}\df{(q^2,-1/z;q^2)_{\infty}}{(q,-q/z;q^2)_{\infty}}
=\df{(1+1/z)(-zq^2,-q^2/z, q^2;q^2)_{\infty}}{(q,-q/z;q^2)_{\infty}}
=\df{(-z,-q^2/z,q^2;q^2)_{\infty}}{z(q,-q/z;q^2)_{\infty}}.
\end{equation}
Putting \eqref{rr6} into \eqref{rr5} and then \eqref{rr5} into \eqref{spad}, we see, upon cancellation, that we have proved \eqref{newprrof1}.
\end{proof}

\section{Applications of Theorem \ref{theorem4}}\label{aysecond1}

\subsection{The case $\a=z$ of Theorem \ref{theorem4} and its corollaries}

\begin{theorem}\label{a=zfunctional}
Let $\omega_1(\a,z;q)$ and $\nu(\a,z;q)$ be defined in \eqref{2b} and \eqref{1}, respectively. Then
\begin{align}\label{a=zfunctionaleqn}
\nu(z,-z;-q)+2q \omega_1(z,-z;-q)=\sum_{n=0}^{\i}\frac{(-q^{n+1};q)_\i q^n}{(-zq^{2n+1};q^2)_\i}.
\end{align}
\end{theorem}

\begin{proof}
Let $\a=z$ in \eqref{gsayeqn1} and then use the identity \cite[p.~25]{AB2}
\begin{equation*}
\sum_{n=0}^{\infty}a^nq^{n^2+2n}=\sum_{n=0}^{\infty}\frac{(-q;q)_na^nq^{n(n+3)/2}}{(-aq^2;q^2)_{n+1}},
\end{equation*}
with $a$ replaced by $z/q$, to transform the series on the extreme right of \eqref{gsayeqn1} to deduce that
\begin{align}\label{intermediateeqn}
\sum_{n=0}^{\i}\frac{(-zq;q)_{2n}(-zq^{2n+2};q^2)_\i q^n}{\left(-q;q\right)_{n}}
&=-\frac{(-zq;q)_{\i}}{\left( -q;q\right)_{\i}}\nu\left(z,-z;-q\right)+2(-zq^2;q^2)_{\infty}\sum_{n=0}^{\infty}z^nq^{n^2+n}.
\end{align}
From Theorem \textup{\ref{newrelation}}, with $z$ replaced by $-z$, $q$ replaced by $-q$, and $\alpha=z$,
\begin{align}\label{relation-z}
\sum_{n=0}^{\infty}z^nq^{n^2+n}=(-zq;q^2)_{\infty}(q;q^{2})_{\infty}\left(\nu{(z, -z; -q)}+q\omega_{1}(z, -z; -q)\right).
\end{align}
Substituting \eqref{relation-z} in \eqref{intermediateeqn} and using $(q;q^2)_{\i}=1/(-q;q)_{\i}$ leads to
\begin{align*}
\sum_{n=0}^{\i}\frac{(-zq;q)_{2n}(-zq^{2n+2};q^2)_\i q^n}{\left(-q;q\right)_{n}}=\frac{(-zq;q)_{\i}}{\left( -q;q\right)_{\i}}\nu\left(z,-z;-q\right)+2q\frac{(-zq;q)_{\i}}{\left( -q;q\right)_{\i}}\omega\left(z,-z;-q\right).
\end{align*}
Now multiply both sides above by $(-q;q)_{\i}/(-zq;q)_{\i}$ to arrive at \eqref{a=zfunctionaleqn}.
\end{proof}

The first identity in the corollary below can be found in Ramanujan's lost notebook \cite[p.~5]{lnb}, \cite[p.~119, Entry 6.3.5]{AB2}. To the best of our knowledge,
the second identity is not given anywhere in the literature.

\begin{corollary}\label{sos}
For any complex numbers $z$ and $q$ such that $|q|<1$,
\begin{align}
 \sum_{n=0}^{\infty}\frac{(-zq;q^2)_nq^n}{(-q;q)_n}
&=2\sum_{n=0}^{\infty}z^nq^{n^2+n}-(q;q^2)_{\infty}(-zq;q^2)_{\infty}\,\nu(z,-z;-q)\label{6.3.5}
\intertext{and}
 \sum_{n=0}^{\infty}\frac{(-zq;q^2)_nq^n}{(-q;q)_n}
&=\sum_{n=0}^{\infty}z^nq^{n^2+n}+(q;q^2)_{\infty}(-zq;q^2)_{\infty}\,q\omega_1(z,-z;-q).\label{10.4b}
\end{align}
\end{corollary}

\begin{proof} Multiply both sides of \eqref{a=zfunctionaleqn} by $ (q;q^2)_{\infty}(-zq;q^2)_{\infty}$ to arrive at
\begin{gather}
(q;q^2)_{\infty}(-zq;q^2)_{\infty}\nu(z,-z;-q)+2q \omega_1(z,-z;-q)\notag\\
=(q;q^2)_{\infty}(-zq;q^2)_{\infty}\sum_{n=0}^{\i}\frac{(-q^{n+1};q)_\i q^n}{(-zq^{2n+1};q^2)_\i}.\label{10.1a}
\end{gather}
Now subtract \eqref{relation-z} from \eqref{10.1a} to deduce \eqref{10.4b}.  Multiply \eqref{relation-z} by 2 and subtract it from \eqref{10.1a} to deduce \eqref{6.3.5}.
\end{proof}


\begin{corollary}\label{a=zfunctionalcoro}  We have
\begin{align}\label{a=zfunctionalcoro1eqn}
\sum_{n=0}^{\i}\frac{(-q^{n+1};q)_\i q^n}{(-zq^{2n+1};q^2)_\i}&-2\notag
q\omega_1(z,-z;-q)\\&=\sum_{n=0}^{\i}\frac{(q;q^2)_n}{(-zq;q^2)_{n+1}}(-z^2)^n(1+zq^{4n+2})q^{3n^2+2n}.
\end{align}
\end{corollary}

\begin{proof}
We appeal to two formulas for $F(a,0;t)$, where we are now using the notation $F(a,b;t)$ in Fine's book \cite{fine}.
First, from  \cite[p.~4, Equation (6.1)]{fine},
\begin{equation}\label{rogers-fine2}
(1-t)F(a,0;t)=\sum_{n=0}^{\infty}\df{(-at)^nq^{n(n+1)/2}}{(tq;q)_n}.
\end{equation}
Second, from \cite[p.~14, Equation (13.3)]{fine} (which is, in fact, a special case of the Rogers--Fine identity),
\begin{equation}\label{rogers-fine}
(1-t)F(a,0;t)=\sum_{n=0}^{\infty}\df{(aq;q)_n}{(tq;q)_n}(-at^2)^n(1-atq^{2n+1})q^{(3n^2+n)/2}.
\end{equation}
In  \eqref{rogers-fine} and \eqref{rogers-fine2}, first replace $q$ by  $q^2$.  Secondly, set  $a=1/q $, and  $t=-zq $.  Combining the two resulting identities,  we find that, respectively, by \eqref{1},
\begin{align}
F\left(\df{1}{q},0;-zq\right)&=\sum_{n=0}^{\infty}\df{z^nq^{n(n+1)}}{(-zq;q^2)_{n+1}}=\nu(z,-z;-q)\notag\\
&=\sum_{n=0}^{\i}\frac{(q;q^2)_n}{(-zq;q^2)_{n+1}}(-z^2)^n(1+zq^{4n+2})q^{3n^2+2n}.\label{rogers-fine4}
\end{align}
Now substitute the far-right side of \eqref{rogers-fine4} into the left-hand side of \eqref{a=zfunctionaleqn} to complete the proof of Corollary \ref{a=zfunctionalcoro}.
\end{proof}

Letting $z=0$ in Corollary \ref{a=zfunctionalcoro} gives
\begin{equation*}
\sum_{n=0}^{\infty}q^n(-q^{n+1};q)_{\infty}=1+2\sum_{n=0}^{\infty}\frac{q^{2n+1}}{(q;q^2)_{n+1}}.
\end{equation*}
This is a disguised form of Euler's theorem, which asserts that the number of partitions of a positive integer $n$ into distinct parts is equal to the number of partitions of $n$ into odd parts. To see this, separate the term for $n=0$ on the left side above, that is, $(-q;q)_{\infty}$,  and then add $1$ to both sides, so that
\begin{equation*}
(-q;q)_{\infty}+\left(1+\sum_{n=1}^{\infty}q^n(-q^{n+1};q)_{\infty}\right)=2\left(1+\sum_{n=0}^{\infty}\frac{q^{2n+1}}{(q;q^2)_{n+1}}\right).
\end{equation*}
Now,
 $$1+\sum_{n=1}^{\infty}q^n(-q^{n+1};q)_{\infty}=(-q;q)_{\infty},$$
 since the series on the left-hand side enumerates partitions into distinct parts with the smallest part being $n$. Next, $$1+\sum_{n=0}^{\infty}\frac{q^{2n+1}}{(q;q^2)_{n+1}}$$
  enumerates partitions into odd parts with $2n+1$  being the largest odd part.

\begin{corollary}\label{alpha=z=-1}
For $|q|<1,$
\begin{align}\label{alphaagain}
\sum_{n=0}^{\i}\frac{(-q^{n+1};q)_\i q^n}{(q^{2n+1};q^2)_\i} -2\sum_{n=0}^{\i}\frac{q^{2n+1}}{(q;q^2)^2_{n+1}} = \sum_{n=0}^{\i}(-1)^nq^{3n^2+2n}(1+q^{2n+1}).
\end{align}
\end{corollary}

\begin{proof}
Put $z=-1$ in \eqref{a=zfunctionalcoro1eqn}.
\end{proof}

We next define three partition functions.

The number of \emph{overpartitions} of a positive integer $n$ is equal to the number of partitions of $n$ arranged in decreasing order, where the first appearance of a part may or may not be overlined.  Thus, $3, \overline{3}, 2+1, \overline{2}+1, 2+\overline{1}, \overline{2}+\overline{1}, 1+1+1,$ and $\overline{1}+1+1$ are the eight overpartitions of 3.

Let $p^*(n)$ denote the the number of overpartitions of $n$ into \emph{distinct} non-negative parts, such that all odd parts less than twice the smallest part are overlined and such that all of the even parts are  overlined.

Let $p_*(n)$ denote the number of weighted partitions of $n$ into odd parts, where, if $\l_1,\l_2,\cdots,\l_k$ are the distinct parts of the partition with $\l_k>\l_{k-1}>\cdots>\l_2>\l_1$, then the weight of the partition is equal to $\prod_{i=1}^{k}w(\l_i)$, where
\begin{equation}\label{star}
w(\l_i)=\begin{cases}
1+\text{the number of appearances of } \lambda_i,\quad &\text{if}\hspace{1mm} 1\leq i<k,\\
\text{the number of appearances of } \lambda_i, &\text{if}\hspace{1mm} i=k.
\end{cases}
\end{equation}

\begin{example}\label{p666}
Let $n=5$. We first find $p^*(5)$. Note that $0$ and $\overline{0}$ may also be parts of the overpartitions enumerated by $p^*(5)$. Hence the admissible partitions are $\overline{5}, 5+\overline{0}, \overline{5}+\overline{0}, \overline{4}+\overline{1}, \overline{4}+1+\overline{0}, \overline{4}+\overline{1}+\overline{0}, \overline{3}+\overline{2}, 3+\overline{2}+\overline{0}, \overline{3}+\overline{2}+\overline{0}, 3+1+1+\overline{0}, \overline{3}+1+1+\overline{0}, 3+\overline{1}+1+\overline{0}, \overline{3}+\overline{1}+1+\overline{0}, \overline{2}+1+1+1+\overline{0}, \overline{2}+\overline{1}+1+1+\overline{0}, 1+1+1+1+1+\overline{0}, \overline{1}+1+1+1+1+\overline{0}$. According to its definition,
\begin{equation*}
p^*(5)=17.
\end{equation*}
Next, we calculate $p_*(5)$. The admissible partitions are $5, 3+1+1$ and $1+1+1+1+1$. Hence
\begin{equation*}
p_*(5)=1+(1)(3)+5=9.
\end{equation*}
\end{example}
Corollary \ref{alpha=z=-1} gives an interesting identity involving $p^{*}(n)$ and $p_{*}(n)$, which  is in the spirit of Euler's pentagonal number theorem.

\begin{theorem}\label{pntanalogue}
Let $p^*(n)$ and $p_*(n)$ be defined as above. For $n \geq 1$,
\begin{align*}
p^*(n) -2 p_*(n) =
\left\{
	\begin{array}{ll}
		(-1)^j,  & \mbox{if } n = 3j^2+2j\,\, or \,\, 3j^2+4j+1, \\
		0, & \mbox{otherwise. }
	\end{array}
\right.
\end{align*}
\end{theorem}

\begin{proof}
We first show that
\begin{align}\label{p^*n}
\sum_{m=0}^{\infty}p^*(m)q^m =\sum_{n=0}^{\i}\frac{(-q^{n+1};q)_\i q^n}{(q^{2n+1};q^2)_\i}.
\end{align}
Rewrite the right-hand side of \eqref{p^*n} in the form
\begin{equation}\label{notag6}
\sum_{n=0}^{\i}q^n(-q^{n+1};q)_n(-q^{2n+2};q^2)_\i\frac{(-q^{2n+1};q^2)_\i }{(q^{2n+1};q^2)_\i}.
\end{equation}
Let $n$ denote the smallest part in an overpartition of a positive integer $m$. First,
\begin{equation}\label{notag4}
q^n(-q^{n+1};q)_n(-q^{2n+2};q^2)_\i
 \end{equation}
 generates partitions, wherein the parts are distinct and each odd part is less than twice the smallest part. Since these parts are distinct, we can overline each of the elements in any given partition.  Secondly,
\begin{equation}\label{notag5}
\frac{(-q^{2n+1};q^2)_\i }{(q^{2n+1};q^2)_\i}
\end{equation}
generates partitions into odd parts, which may or may not be overlined.
Taking the interpretations of \eqref{notag4} and \eqref{notag5} together in \eqref{notag6}, we see that we have generated the overpartitions enumerated by $p^{*}(m)$ from the right-hand side of  \eqref{p^*n}, i.e., we have concluded our proof of \eqref{p^*n}.

Next, consider the second series on the left-hand side of \eqref{alphaagain}.  
 Clearly, $\displaystyle\frac{q^{2n+1}}{(q;q^2)^2_{n+1}}$ generates partitions into odd parts with the largest part $2n+1$. Write
\begin{align*}
&\frac{q^{2n+1}}{(q;q^2)^2_{n+1}}
=\frac{1}{(1-q)^2}\cdot\frac{1}{(1-q^3)^2}\cdots\frac{1}{(1-q^{2n-1})^2}\cdot\frac{q^{2n+1}}{(1-q^{2n+1})^2}\nonumber\\
&=\sum_{k_1,k_2,\cdots,k_{n}=0}^{\infty}(k_1+1)(k_2+1)\cdots(k_{n}+1)q^{1\cdot k_1+3.k_2+\cdots+(2n-1)\cdot k_n}\sum_{k_{n+1}=1}^{\infty}k_{n+1}q^{(2n+1)\cdot k_{n+1}}.
\end{align*}
Recalling the definition of $p_{*}(n)$ in \eqref{star}, we see that
\begin{equation}\label{p_*n}
\sum_{m=0}^{\infty}p_*(m)q^m=\sum_{n=0}^{\i}\frac{q^{2n+1}}{(q;q^2)^2_{n+1}}.
\end{equation}
Theorem \ref{pntanalogue} now follows from \eqref{p^*n} and \eqref{p_*n} by equating the coefficient of $q^n$ on both sides of Corollary \ref{alpha=z=-1}.
\end{proof}
\begin{example}
Let $n=5$. From Example \ref{p666},  $p^*(5)=17$ and $p_*(5)=9$, and so $p^*(5)-2p_*(5)=-1$ as predicted by Theorem \textup{\ref{pntanalogue}}, since $5$ is of the form $3j^2+2j$, precisely when $j=1$.

 Next, let $n=6$. Note that $6$ cannot be written in the form $3j^2+4j+1$ or $3j^2+2j$. After counting all of the relevant partitions, we find that $p^*(6)=28$ and $p_*(6)=14$. Thus, $p^*(6)-2p_*(6)=0$ as predicted by Theorem \textup{\ref{pntanalogue}}.
\end{example}

Theorem \ref{pntanalogue} readily gives the following corollary determining the parity of $p^*(n)$.

\begin{corollary}
For $n \geq 1$,
\begin{align*}
p^*(n)  \equiv
\left\{
	\begin{array}{ll}
		1~ (\textup{mod} ~2), & \mbox{if } n = 3j^2+2j~or~ 3j^2+4j+1, \\
		0 ~ (\textup{mod} ~2), & \mbox{otherwise. }
	\end{array}
\right.
\end{align*}
\end{corollary}

\subsection{The case $z=-1/q$ of Theorem \ref{theorem4}}

Setting $z=-1/q$ in Theorem \ref{theorem4}, we deduce the following corollary. See also Kang's paper \cite[Corollary 7.4]{kang}.
\begin{corollary}
\begin{align*}
\sum_{n=0}^{\i}\frac{\a^nq^{n(n+1)/2}}{(-\a q;q)_{n+1}} = 1.
\end{align*}
\end{corollary}

\subsection{The case $\a=-z$ of Theorem \ref{theorem4} and its corollaries}

\begin{corollary}\label{alpha=-z} We have
\begin{align}\label{alpha=-zeqn}
\nu(z,-z;-q)=\sum_{n=0}^{\i}\frac{(q^{n+1};q)_\i}{(-zq^{2n+1};q^2)_\i}q^n.
\end{align}
\end{corollary}

\begin{proof}
Let $\a= -z$ in Theorem \ref{theorem4} and note that the series on the extreme right-hand side vanishes. Corollary \ref{alpha=-z} now follows readily after some simplification.
\end{proof}

\begin{corollary}
We have
\begin{equation*}
\omega(z,-z;-q)=\frac{1}{2q}\sum_{n=0}^{\i}\frac{\left\{(-q^{n+1};q)_\i-(q^{n+1};q)_\i\right\} q^n}{(-zq^{2n+1};q^2)_\i}.
\end{equation*}
\end{corollary}

\begin{proof}
Use Corollary \textup{\ref{alpha=-z}} in Theorem \textup{\ref{a=zfunctional}}.
\end{proof}

\begin{corollary}\label{aplha=-z=-q} We have
\begin{align}\label{aplha=-z=-qeqn}
\sum_{n=1}^{\i}\frac{q^n}{(-q^{n};q)_\i}&=1-(q;q^2)_\i,
\end{align}
\end{corollary}

\begin{proof}
Let $z=-q$ in \eqref{alpha=-zeqn}, so that
\begin{align}\label{corb}
\nu(-q,q;-q)=\sum_{n=0}^{\i}\frac{q^n}{(-q^{n+1};q)_\i}=\df{1}{q}\sum_{n=1}^{\infty}\df{q^n}{(-q^n;q)_{\infty}}.
\end{align}
Recall  Euler's theorem \cite[p.~19, Corollary 2.2, Equation (2.2.6)]{gea}, with $q$ replaced by $q^2$, that is,
\begin{align*}
\sum_{n=0}^{\i}\frac{(-z)^nq^{n^2-n}}{(q^2;q^2)_n}=(z;q^2)_\i,
\end{align*}
Set $z=q$ above to deduce that
\begin{equation}\label{cora}
(q;q^2)_{\infty}-1=\sum_{n=1}^{\infty}\df{(-1)^nq^{n^2}}{(q^2;q^2)_n}=-q\nu(-q,q;-q),
\end{equation}
by \eqref{1}. Combine \eqref{corb} and \eqref{cora} to deduce \eqref{aplha=-z=-qeqn}.
\end{proof}

\subsection{The case $\a=q$ of Theorem \ref{theorem4} and its corollaries}

\begin{theorem}\label{alpha=q}
For $z \in \mathbb{C}$ and $|q|<1$,
\begin{align}
&\sum_{n=0}^{\i}q^n(-zq^n;q)_{n+1}(-zq^{2n+2};q^2)_\i\nonumber\\
&= -\frac{q}{z}\nu\left(\frac{q^2}{z},-\frac{q^2}{z};-q\right)
 +\frac{1}{q}(-z;q^2)_\i \left( -1+\frac{(-q;q)_\i(q^2;q^2)_\i(-q^2/z;q^2)_\i}{(-q^3/z;q^2)_\i}    \right).\label{alpha=qeqn}
\end{align}
\end{theorem}

\begin{proof}
Putting $\a =q$  in \eqref{gsayeqn1}, we find that
\begin{align*}
\sum_{n=0}^{\i}\frac{(-zq;q)_{2n}(-zq^{2n+2};q^2)_\i q^n}{\left( -z;q\right)_{n}}   =& -\frac{q}{z}\frac{(-zq;q)_{\i}}{\left( -z;q\right)_{\i}}\nu\left(\frac{q^2}{z},-\frac{q^2}{z};-q\right)\nonumber\\
&+(-zq^2;q^2)_{\infty}\sum_{n=0}^{\infty}\frac{\left(-q/z;q\right)_{n+1}q^{n(n+3)/2}}{\left(- q^3/z;q^2\right)_{n+1}}\\
=&-\frac{q}{z}\frac{(-zq;q)_{\i}}{\left( -z;q\right)_{\i}}\nu\left(\frac{q^2}{z},-\frac{q^2}{z};-q\right)\\
&+\frac{1}{q}(-zq^2;q^2)_{\infty}\left(-1+\sum_{n=0}^{\infty}\frac{\left(-q/z;q\right)_{n}q^{n(n+1)/2}}{\left(- q^3/z;q^2\right)_{n}}\right).
\end{align*}
Letting $a=q$ and $b=-q/z$ in \eqref{gauss}, and then using the resulting identity on the right-hand side above, we deduce that
\begin{align*}
\sum_{n=0}^{\i}\frac{(-zq;q)_{2n}(-zq^{2n+2};q^2)_\i q^n}{\left( -z;q\right)_{n}}=&-\frac{q}{z}\frac{(-zq;q)_{\i}}{\left( -z;q\right)_{\i}}\nu\left(\frac{q^2}{z},-\frac{q^2}{z};-q\right)\\
&+\frac{1}{q}(-zq^2;q^2)_{\infty}\left( -1+\frac{(-q;q)_\i(q^2;q^2)_\i(-q^2/z;q^2)_\i}{(-q^3/z;q^2)_\i}    \right).
\end{align*}
Now multiply both sides by $(1+z)$ to deduce \eqref{alpha=qeqn}.
\end{proof}

Ramanujan's third order mock theta function $\phi(q)$ is defined by
\begin{align}\label{phiq}
\phi(q):=\sum_{n=0}^{\infty} \frac{q^{n^2}}{(-q^2;q^2)_n}.
\end{align}
A corollary of Theorem \ref{alpha=q} gives the corrected version of \cite[Theorem 4.2]{ady1} involving $\phi(q)$.

\begin{corollary}\label{correctedphi}
With $\phi(q)$ defined by \eqref{phiq},
\begin{align}\label{phiqa}
\sum_{n=0}^{\i}q^n(-q^n;q)_{n}(-q^{2n+1};q^2)_\i =1- \phi(q)+(q^2;q^2)_\i (-q;q^2)^{3}_\i.
\end{align}
\end{corollary}

\begin{proof}
Let $z=q$ in Theorem \ref{alpha=q} and simplify.
\end{proof}

Corollary \ref{correctedphi} has the following companion.

\begin{corollary}\label{z=qTh10.12}  We have
\begin{align}\label{eqnz=qTh10.12}
\sum_{n=0}^{\i}q^n(-q^n;q)_{n}(-q^{2n+1};q^2)_\i&= 1+2q~\nu(q,-1;q).
\end{align}
\end{corollary}

\begin{proof}
Employing \eqref{fromreciprocity} with $\a=1/q$ and $z=1/q$, we find that
\begin{align}\label{slight}
\nu(1/q,1/q;q)+q\nu(q,-1;q)&= \frac{(-q,-q,q^2;q^2)_\i}{(-1,q;q^2)_\i}.
\end{align}
Using \eqref{1}, we rewrite \eqref{slight} in the slightly modified form
\begin{align*}
\sum_{n=0}^{\infty}\frac{q^{n^2}}{(-q^2;q^2)_n}+2q\nu(q,-1;q)&=\frac{(-q;q^2)^2_\i(q^2;q^2)_\i}{(-q^2;q^2)_\i(q;q^2)_\i}.
\end{align*}
Using \eqref{phiq} and simplifying, we deduce that
\begin{align*}
\phi(q)+2q~\nu(q,-1;q) =(-q;q^2)^3_\i(q^2;q^2)_\i.
\end{align*}
Use the identity above in \eqref{phiqa} to deduce \eqref{eqnz=qTh10.12}.
\end{proof}

Let $p'(n)$ denote the number of partitions of $n$ into non-negative parts in which the smallest part can occur at most twice, all other parts must be distinct, and in which all even parts are strictly less than twice the smallest part. For example, if $n=5$, then the admissible partitions are $5$, $5+0$, $5+0+0$ and $3+2$. Hence $p'(5)=4$. One can check that $\{p'(n)\}_{n=0}^{25}=\{1, 2, 2, 2, 4, 4, 4, 6, 6, 8, 10, 10, 12, 14, 16, 18, 22, 24, 26, 32, 34, 38, 44, 48, 54, 62\}$. This leads us to suggest that $p'(n)$ for any $n\geq1$ might be even. That this is the case is an immediate consequence of Corollary \ref{z=qTh10.12}. This is given below.

\begin{corollary}\label{even}
Let $p'(n)$ be defined as above. Then $p'(0)=1$ and $p'(n)$ is even for any $n\geq1$.
\end{corollary}
\begin{proof}
We first show that
\begin{equation}\label{gfeven}
\sum_{n=0}^{\infty}p'(n)q^n=\sum_{n=0}^{\i}q^n(-q^n;q)_{n}(-q^{2n+1};q^2)_\i.
\end{equation}
This is easy to see once we regard the index of summation $n$ in the series on the right-hand side as denoting the smallest part in a partition generated by it. Then the summand is given by
\begin{equation*}
q^n(-q^n;q)_{n}(-q^{2n+1};q^2)_\i=(q^n+q^{2n})(-q^{n+1};q)_{n-1}(-q^{2n+1};q^2)_\i.
\end{equation*}
Clearly, this implies that the smallest part $n$ can occur at most twice. Also, all other parts must be distinct, and no even part is as large as twice the smallest part. This proves \eqref{gfeven}.

Now the result follows immediately by invoking Corollary \eqref{z=qTh10.12} and noting that there is a $2$ in front of $q~\nu(q,-1;q)$ on its right-hand side.
\end{proof}

\begin{corollary}\label{corc} We have
\begin{align*}
\nu(1/q,1;-q)= \sum_{n=0}^{\i}\frac{q^{n^2}}{(q;q^2)_{n+1}}=1+2q\sum_{n=0}^{\i}\frac{q^{n^2+2n}}{(q;q^2)_{n+1}}=1 +2q~\nu(q,-1;q).
\end{align*}
\end{corollary}

\begin{proof} If we put $z=1/q$ in \eqref{newprrof1}, we find that the left-hand side of \eqref{newprrof1} is the same as the left-hand side of \eqref{eqnz=qTh10.12}.
\end{proof}

\begin{corollary}  We have
\begin{equation}\label{alpha=q, z=-1}
\sum_{n=0}^{\infty}q^n(q^{n};q)_{n+1}(q^{2n+2};q^2)_{\infty}=-\frac{1}{q}\sum_{n=1}^{\infty}\frac{(-1)^nq^{n^2+n}}{(-q^3;q^2)_n}.
\end{equation}
\end{corollary}

\begin{proof}
Let $z=-1$ in Theorem \ref{alpha=q}. This leads to \eqref{alpha=q, z=-1} upon simplification.
\end{proof}

\begin{corollary}  We have
\begin{gather*}
\sum_{n=0}^{\i}q^n(q^{n+2};q)_{n+1}(q^{2n+4};q^2)_\i= \frac{1}{q}\sum_{n=0}^{\infty}\df{(-1)^nq^{n(n+1)}}{(q;q^2)_{n+1}}
 -\frac{1}{q}(q^2;q^2)_\i. 
\end{gather*}
\end{corollary}

\begin{proof} Set $z=-q^2$ in Theorem \ref{alpha=q}, or set
$z=-1$ in \eqref{pnnz}.
\end{proof}

\section{Concluding remarks}\label{cr}
In this paper, we have partially answered the question posed by Li and Yang in their paper \cite{liyangmock} and obtained generalizations of the Andrews-Yee identities for $\omega_1(\a, z;q)$ and $\nu(\a,z;q)$, defined in \eqref{2b} and \eqref{1}, respectively. While $\nu(\a,z;q)$ is essentially equivalent to their trivariate generalization of $\nu(q)$, the function $\omega_1(\a, z;q)$  is different from the trivariate generalization of $\omega(q)$ that they consider in their paper.

This raises an important question:  Are there analogues of Theorems \ref{theorem3} and \ref{theorem4} for the trivariate versions considered
by Li and Yang? Especially, in view of \eqref{omegaG}, obtaining such an analogue for $\omega_0(y,z;q)$ would be crucial, since $G(a,b;q)$ is related to the universal mock theta function $g_3(x;q)$. This appears to be elusive so far. However, one might begin by first obtaining an analogue of our functional relation in Theorem \ref{newrelation}, if it exists.

Also, it appears difficult to obtain a generalization of \eqref{pwnz} of the type given in Theorem \ref{theorem4}, and similarly, a generalization of \eqref{pnnz} of the type given in Theorem \ref{theorem3}. However, seeking these results definitely seems to be worthy of merit in view of their applications.

\begin{center}
\textbf{Acknowledgements}
\end{center}

The authors thank Shane Chern and Ae Ja Yee for their important suggestions for our paper. The second and the third authors sincerely thank the SPARC project SPARC/2018-2019/P567/SL for funding their stay at the University of Illinois at Urbana-Champaign in January 2020 and January-May 2020, respectively, where part of this work was carried out. They also thank the University of Illinois at Urbana-Champaign for its hospitality.

\end{document}